\documentclass[ECP]{ejpecp} 

\usepackage[T1]{fontenc}
\usepackage[utf8]{inputenc}

\usepackage{stmaryrd}  

\SHORTTITLE{Explosivity in 1-d Activated Random Walk}
\TITLE{Explosivity in 1-d Activated Random Walk\support{C.H.\ was partially supported by NSF grant DMS-2503778. T.J.\ was partially
supported by NSF grant DMS-2503779 and gratefully acknowledges support
from Uppsala University and the Wenner--Gren Foundation. J.R.\ was partially
supported by a Simons Foundation Targeted Grant awarded to the R\'enyi
Institute. L.R.\ was partially supported by FAPESP grants 2023/13453-5
and 2025/27064-6, and CNPq grant 408590/2024-6.}}

\AUTHORS{%
  Nicolas~Forien\footnote{CEREMADE, CNRS, Université Paris-Dauphine, Université PSL, 75016 Paris, France.\\
  \EMAIL{forien@ceremade.dauphine.fr}}\orcid{0000-0002-3146-644X}
  \and
  Christopher~Hoffman\footnote{University of Washington, United States of America.
  \EMAIL{choffman@uw.edu}}
  \and
  Tobias~Johnson\footnote{College of Staten Island (CUNY), United States of America.
    \EMAIL{tobias.johnson@csi.cuny.edu}}\orcid{0000-0002-8167-9371}
  \and 
  Josh Meisel\footnote{CUNY Graduate Center, United States of America.
    \EMAIL{jmeisel@gradcenter.cuny.edu}}\and
  Jacob~Richey\footnote{The Alfr\'ed R\'enyi Institute of Mathematics, Hungary.
    \EMAIL{jfrichey001@gmail.com}}\orcid{0000-0001-9776-3343}
    \and
    Leonardo~T.~Rolla\footnote{Universidade de São Paulo, Brazil.
    \EMAIL{leonardo.rolla@gmail.com}}\orcid{0000-0002-2084-4687}
    }


\AMSSUBJ{60K35} 
\AMSSUBJSECONDARY{82C22} 

\SUBMITTED{January 8, 2026}
\ACCEPTED{August 23, 2026}

\ARXIVID{2601.03411} 

\VOLUME{31}
\YEAR{2026}
\PAPERNUM{66}
\DOI{10.1214/26-ECP806}

\ABSTRACT{We show that Activated Random Walk on $\mathbb{Z}$ is \emph{explosive} above criticality. That
is, activating a single particle in a supercritical state of sleeping particles triggers an infinite
avalanche of activity with positive probability. This extends the same result recently proven by Brown,
Hoffman, and Son for i.i.d.\ initial distributions to the setting of ergodic ones, thus completing the
proof of a conjecture of Rolla's in dimension one. As a corollary we obtain that, for supercritical ergodic
initial distributions with any positive density of particles initially active, the system will stay active
almost surely. Our result is another piece of evidence attesting to the universality of the phase
transition of Activated Random Walk on $\mathbb{Z}$.}

\begin{document}


\section{Introduction}
\label{sec1}

Activated Random Walk (ARW) is a stochastic interacting particle system
consisting of indistinguishable particles which occupy the sites of a graph.
The particles, each of which can be {\em{active}} or {\em{sleeping}}, move
according to the following (local) rules: Active particles perform continuous-time
simple symmetric random walks at rate $1$ and fall asleep at rate
$\lambda > 0$ while alone. Sleeping particles do not move and become active
whenever they are visited by an active particle. ARW is a close cousin
of the Stochastic Sandpile, itself a stochastic version of the Abelian
Sandpile. These processes (and many more) were developed as mathematical
models of \emph{self-organized criticality (SOC)}, a statistical phenomenon
observed in complex real-world systems. SOC aims at providing a heuristic
and statistical description of natural systems, like forest fires, mountain
landscapes, and earthquakes, and of some artificial ones, like stock markets
and power grids, which evolve in time, exhibit cascading kinetic avalanches
on all scales, and fluctuate around a universal critical state without
external tuning \cite{BakTangWiesenfeld87,watkins2016twentyfive}.

Part of the core philosophy of SOC is universality: the same features of
SOC should be visible for many different models. In the context of ARW,
we expect the same critical state to appear in many different formulations
of the model (see \cite{levine2023universality} for an overview). One aspect
is that the self-organized critical state arising for the
\emph{driven-dissipative} version of the model, where particles are added
to the interior but deleted at the boundary of a finite box, is thought
to be the same as the critical state in a traditional absorbing-state phase
transition for the \emph{fixed-energy} version on an infinite lattice, as
originally proposed by Dickman, Mu\~noz, Vespignani, and Zapperi
\cite{dickman2000paths}. In this fixed-energy model, the system is initialized
with a random configuration of particles
$\sigma \colon \mathbb{Z}^{d} \to \mathbb{N}\cup \{\mathfrak{s}\} =
\{0, \mathfrak{s}, 1, 2, \ldots \}$ whose distribution is ergodic with
average particle density $\rho =\mathbf{E}\mathopen{\lvert }\sigma (0)\mathclose{\rvert }$. (Here
$\mathfrak{s}$ denotes the presence of a single sleeping particle, and
we let $\mathopen{\lvert }\mathfrak{s}\mathclose{\rvert } = 1$ so that $\mathopen{\lvert }\sigma (x)\mathclose{\rvert }$ counts the number
of particles at site $x$. Also, ergodicity is always with respect to the
group of translations of $\mathbb{Z}^{d}$.) We say that an instance of
this system \emph{fixates} if each vertex of $\mathbb{Z}^{d}$ is visited
by an active particle finitely many times, and otherwise that it
\emph{stays active}. Fixation is a zero-one event
\cite[Lemma~4]{rolla2012absorbing}. For a stochastically increasing family
of initial distributions, like i.i.d.\ $\mathrm{Poisson}(\rho )$, it is
a consequence of monotonicity that there exists some critical density
$\rho _{c}$ such that ARW fixates if $\rho <\rho _{c}$ and stays active
if $\rho >\rho _{c}$.

One sign of universality is that this phase transition, which a priori
depends on the family of initial distributions (as well as the sleep rate
$\lambda $ and dimension $d$), in fact depends only on their densities.
That is, there exists a single $\rho _{c}=\rho _{c}(\lambda ,d)$ such that
for \emph{any} ergodic initial distribution with density $\rho $, the system
fixates if $\rho <\rho _{c}$ and stays active if $\rho >\rho _{c}$, as
shown by Rolla, Sidoravicius, and Zindy
\cite{RollaSidoraviciusZindy19}. (Notably, the behavior of the system is
not known at criticality, nor is it known whether it depends on the initial
distribution.) This result applies only when the system starts with all
particles active. Rolla conjectured that this requirement could be dropped.
He also conjectured the stronger statement that at supercritical density
with all but one particle asleep, ARW remains active with positive probability
\cite[p.~517]{rolla2020activated}.

Recently, Dickman et al.'s prediction that the density of the driven-dissipative
model's stationary distribution on a box converges to $\rho _{c}$ as the
box size grows, known as the \emph{density conjecture}, was proven in dimension
one \cite{HoffmanJohnsonJunge24}. Using the same technology, Brown, Hoffman,
and Son proved Rolla's conjecture in dimension one, restricted to the case
of an i.i.d.\ initial distribution of particles \cite{explosivity}. Our
contribution here is to remove this restriction on the initial distribution,
proving the conjecture in dimension one for all ergodic initial distributions,
with a shorter and conceptually simpler proof.

To state our results, call a configuration
$\sigma \colon \mathbb{Z}\to \mathbb{N}\cup \{\mathfrak{s}\}$
\emph{explosive} if ARW stays active (i.e.,\ each site is visited i.o.)\ with
positive probability when started from $\sigma $ with particles activated
on some nonvacant site. We note that if ARW has positive probability of
staying active starting from $\sigma $ when particles are activated at
some given nonvacant site $x$, then the same is true for \emph{any} nonvacant
site $y$, since with positive probability a particle at $y$ marches to
$x$ and back without falling asleep. Furthermore, if $\sigma $ has active
particles, then ARW has positive probability of remaining active from
$\sigma $ without waking anything.

\begin{theorem}%
\label{thm:explosive}
Let $\sigma $ be a random element in
$(\mathbb{N}\cup \{\mathfrak{s}\})^{\mathbb{Z}}$ with an ergodic distribution
and density $\mathbf{E}\mathopen{\lvert }\sigma (0)\mathclose{\rvert } = \rho $. Then $\sigma $ is a.s.\ explosive
if $\rho > \rho _{c}$ and a.s.\ not explosive if $\rho < \rho _{c}$.
\end{theorem}

In particular, if $\rho >\rho _{c}$ and the initial distribution is not
supported on stable configurations, then ARW remains active a.s.\ by the
zero-one law \cite[Lemma~4]{rolla2012absorbing}.

We mention that we assume that the underlying random walks in our model
are symmetric, nearest-neighbor random walks. In addition to holding in
any dimension, Rolla, Sidoravicius, and Zindy's original result on the
universality of the critical density allows for a wider class of random
walks (with the critical density depending on the walk). While the density
conjecture has been proven for biased random walks
\cite{forienDensity}, the exponential upper tail bound for density after
stabilization
\cite[Proposition~7.1 and Theorem~8.4]{HoffmanJohnsonJunge24} is used in
this paper in an essential way via \cite[Proposition~4.1]{cutoff} but remains
unproven for the biased case.

\begin{remark}
\label{rem1.2}
For ARW on regular trees, Sidoravicius, Rolla, and Zindy's result also
holds, establishing that there is a critical density separating fixation
and activity for ergodic initial distributions with all particles active
\cite[Section~12.2]{rolla2020activated}. This critical density is known
to be strictly between $0$ and $1$
\cite[Theorem~1.6]{StaufferTaggi18}. It is proven in
\cite{JohnsonRolla19} that for ARW with sleep rate $\lambda =0$, there
exist i.i.d.\ initial distributions of sleeping particles of arbitrarily
high density such that every site is visited finitely many times a.s.\ when
a single particle is activated. By monotonicity in the sleep rate, we can
conclude that for ARW with any sleep rate $\lambda >0$ on a regular tree,
there exist i.i.d.\ distributions of sleeping particles with supercritical
density that are \emph{not} explosive, in contrast with Theorem~\ref{thm:explosive}.

When activating a single particle starting from such an initial configuration,
we have local fixation a.s., but it is unclear if the system fixates globally
as it would for a nonexplosive configuration on a lattice. On a tree, it
is plausible that there is a regime where activity ceases locally but persists
somewhere on the graph for all time, as occurs for the contact process
on trees \cite{pemantle1992contact}.
\end{remark}

\section{Proofs}
\label{sec2}

We use the notation $\mathopen{\llbracket }a,b \mathclose{\rrbracket } := [a,b] \cap \mathbb{Z}$ for any
$a,b \in \mathbb{Z}$. We use the site-wise construction of ARW, in which
each site contains a stack of jump and sleep instructions, and an active
particle follows the next instruction on its stack after exponentially
distributed waiting times. If the instructions are i.i.d., with sleep instructions
occurring with probability $\lambda /(1+\lambda )$ and jump instructions
equally likely to be to the left or right, the resulting process matches
the continuous-time Markov chain described in the introduction. See
\cite{rolla2020activated} for a more formal description.

A configuration is said to be \emph{stable} on
$V \subseteq \mathbb{Z}$ if it contains no active particles on $V$. If
there is an active particle at site~$v$, then it is \emph{legal} to
\emph{topple} $v$, which means executing the next instruction on the stack
at $v$ and updating the configuration accordingly. The
\emph{odometer} for a finite sequence of topplings within a set of vertices
$V$ is the function $V\to \mathbb{N}$ counting the number of times each
site is toppled. The \emph{stabilizing odometer} of a configuration
$\sigma $ on a set $V$ is the site-wise supremum of the odometers of all
finite sequences of legal topplings within $V$. When $V$ is finite, the
stabilizing odometer is a.s.\ finite at every site, and corresponds to
a sequence of topplings that leaves a stable configuration on $V$ (by the
abelian property, all such sequences will have the same odometer). When
the stabilizing odometer is not finite at every site, the term is technically
a misnomer, since the configuration is not locally stabilized. Fixation
of ARW on $\mathbb{Z}$ starting from configuration $\sigma $ is equivalent
to the statement that the stabilizing odometer for $\sigma $ on
$\mathbb{Z}$ is finite at all vertices.

For odometers $u$ and $u'$, we denote pointwise domination by
$u\leq u'$. We do the same for particle configurations under the ordering
$0 < \mathfrak{s}< 1 < 2 \ldots $. We say that the subset
$U \subseteq V$ is \emph{visited} during the stabilization of
$\sigma $ on $V$ if the stabilizing odometer is positive everywhere on
$U$, or, equivalently, if at each site of $U$ an active particle is present
at some time during the process. For a given configuration $\sigma $ and
set of vertices $U$, let $\mathsf{A}^{U}\sigma $ denote the configuration
obtained from $\sigma $ by waking all particles in $U$.

The sketch of the proof of Theorem~\ref{thm:explosive} goes like this.
To show activity starting from $\mathsf{A}^{\{0\}}\sigma $, it will be
enough to show that every site in $\mathbb{Z}$ is visited. Begin the stabilization
by toppling one of the particles at the origin until it falls asleep, assuming
there is at least one. Suppose that we get lucky and visit a large interval
$\mathopen{\llbracket }-n,n \mathclose{\rrbracket }$. Knowing that all particles on this interval eventually become
active, their supercritical density together with
\cite[Proposition~4.1]{cutoff} imply that it is exponentially likely we
visit $n+1$ and $-n-1$. Now, knowing that all particles on
$\mathopen{\llbracket }-n-1,n+1\mathclose{\rrbracket }$ eventually become active, we repeat the argument to conclude
that it is exponentially likely we visit $n+2$ and $-n-2$. Continuing like
this, we conclude that given our lucky first particle visiting
$\mathopen{\llbracket }-n,n \mathclose{\rrbracket }$, all of $\mathbb{Z}$ is visited with probability at least
$1-Ce^{-cn}$ for some constants $c,C>0$. We only need to choose $n$ large
enough to make this quantity positive to complete the proof.

One tool we need in carrying out this plan is the
\emph{preemptive abelian property}, which states that if a sleeping particle
will eventually be woken, we can wake it from the beginning without altering
the stabilizing odometer of the system. We note that this is a deterministic
statement (and actually holds for any graph).
%
\begin{lemma}%
\label{lem:preemptive}
Let $V\subseteq \mathbb{Z}$, and suppose that $U\subseteq V$ is a collection
of sites visited during the stabilization of $\sigma $ on $V$. Then stabilizing
$\sigma $ or $\mathsf{A}^{U}\sigma $ on $V$ yields the same odometer.
\end{lemma}
\begin{proof}
This result is already proven when $V$ is finite
\cite[Corollary~3.2]{cutoff}. We extend it to infinite $V$ as follows.
Let $u$ and $u'$ be the odometers stabilizing $\sigma $ and
$\mathsf{A}^{U}\sigma $, respectively, on $V$. We have $u\leq u'$ by monotonicity
\cite[Lemma~2.5]{rolla2020activated}. For the other direction, let
$V_{1}\subseteq V_{2}\subseteq \cdots $ be an arbitrary sequence of finite
sets whose union is $V$. By definition of stabilizing odometer, we have
$u_{i}\nearrow u$ and $u'_{i}\nearrow u'$, where $u_{i}$ and
$u'_{i}$ are the odometers stabilizing $\sigma $ and
$\mathsf{A}^{U}\sigma $, respectively, on $V_{i}$. Thus it suffices to
show $u \ge u'_{i}$ for all $i \ge 1$. Note that only particles awoken
on $V_{i}$ affect $u'_{i}$, so $u'_{i}$ is the odometer stabilizing
$\mathsf{A}^{U \cap V_{i}}\sigma $ on $V_{i}$. Also, $u_{n}$ is positive
on $U \cap V_{i}$ for all large $n$, since it is a finite set and
$u_{n} \nearrow u$. By the finite-graph version of the preemptive abelian
property \cite[Corollary~3.2]{cutoff} then, $u_{n}$ is the odometer stabilizing
$\mathsf{A}^{U \cap V_{i}}\sigma $ on $V_{n}$, and so
$u_{n} \ge u'_{i}$ by monotonicity
\cite[Lemma~2.5]{rolla2020activated}.
\end{proof}

Pertaining to our plan of establishing that activity nucleates out from
$\mathopen{\llbracket }-n,n \mathclose{\rrbracket }$, we define random variables measuring how far activity spreads
once the particles on $\mathopen{\llbracket }-n,n \mathclose{\rrbracket }$ are woken.
%
\begin{definition}[$X(n, \sigma )$, $Y(n, \sigma )$, $X(n, \sigma ,u_{0})$, $Y(n,\sigma ,u_{0})$]
\label{defn2.2}
Given $n\ge 1$ and a particle configuration $\sigma $, we define
$n<X(n, \sigma )\leq \infty $ as the leftmost unvisited site in
$\llbracket n+1,\infty )$ when stabilizing
$\mathsf{A}^{\mathopen{\llbracket }1,n \mathclose{\rrbracket }}\sigma $ on $\mathbb{Z}_{>0}$. We let
$X(n,\sigma )=\infty $ when all sites in $\llbracket n+1,\infty )$ are
visited. Similarly, we define $-\infty \le Y(n, \sigma ) < -n$ as the rightmost
unvisited site in $(-\infty ,-n-1\rrbracket $ when stabilizing
$\mathsf{A}^{\mathopen{\llbracket }-n,-1\mathclose{\rrbracket }}\sigma $ on $\mathbb{Z}_{<0}$.

We will sometimes wish to consider these random variables starting midstream
in an ARW system, i.e., after some instructions have already been executed.
For a given odometer $u \colon \mathbb{Z}\to \mathbb{N}$, we write
$X(n, \sigma , u)$ and $Y(n, \sigma , u)$ to indicate the corresponding
quantities for the system in which the first instruction executed at site~$i$
begins at index $u(i)$ rather than $0$. The initial configuration is still
taken to be $\sigma $.
\end{definition}

\begin{proposition}%
\label{prop:right.avalanche}
For some $\epsilon ,\beta >0$ and some $N\geq 1$, let
$\sigma \colon \mathbb{Z}\to \mathbb{N}\cup \{\mathfrak{s}\}$ be a deterministic
configuration, which for all $n \ge N$ satisfies
%
\begin{align}
\sum _{j=1}^{n} j\mathopen{\lvert }\sigma (j)\mathclose{\rvert } &\geq
\frac{(\rho _{c}+\epsilon )n^{2}}{2}
\label{eq:center.of.mass}
\\
\intertext{and} \sum _{j=1}^{n}\mathopen{\lvert }\sigma (j)\mathclose{\rvert }&\leq \beta n.
\label{eq:max.density}
\end{align}
Then for constants $c,C>0$ depending only on $\epsilon $, $\beta $ and
$\lambda $,
\begin{align*}
\mathbf{P}(X(n, \sigma )<\infty ) &\leq Ce^{-cn},
\end{align*}
for all $n \ge N$.
\end{proposition}
\begin{proof}
Let $E_{k}$ be the event that stabilizing
$\mathsf{A}^{\mathopen{\llbracket }1,k \mathclose{\rrbracket }}\sigma $ on $\mathopen{\llbracket }1,k \mathclose{\rrbracket }$ sends no particle to
$k+1$. We claim that if $X(n, \sigma )<\infty $, then $E_{k}$ holds for
some $k\geq n$. Indeed, if $X(n, \sigma )<\infty $, then
$\mathsf{A}^{\mathopen{\llbracket }1,n \mathclose{\rrbracket }}\sigma $ and
$\mathsf{A}^{\mathopen{\llbracket }1, X(n, \sigma )-1\mathclose{\rrbracket }}\sigma $ have the same stabilizing
odometer on $\mathbb{Z}_{>0}$ by Lemma~\ref{lem:preemptive}, implying that
$E_{X(n, \sigma )-1}$ holds.

Thus we have
\begin{align*}
\mathbf{P}(X(n, \sigma )<\infty ) &\leq \sum _{k=n}^{\infty}
\mathbf{P}(E_{k}).
\end{align*}
By \cite[Proposition~4.1]{cutoff}, from \eqref{eq:center.of.mass} and
\eqref{eq:max.density} for all $k\geq n$ we have
$\mathbf{P}(E_{k})\leq Ce^{-ck}$ for constants $c,C>0$ depending only on
$\epsilon $, $\beta $ and $\lambda $.
\end{proof}

We now state a technical lemma about real sequences, which we use to show
that conditions~\eqref{eq:center.of.mass} and~\eqref{eq:max.density} hold
almost surely for $n$ large enough.

\begin{lemma}%
\label{lem:averages}
Let $a_{1},a_{2},\ldots \in \mathbb{R}$ and suppose that
%
\begin{align}
\label{eq:erg}
\lim _{n\to \infty}\frac{1}{n}\sum _{j=1}^{n}a_{j}=\rho .
\end{align}
Then
%
\begin{align}
\label{eq:weighted}
\lim _{n\to \infty}\frac{1}{n^{2}}\sum _{j=1}^{n}ja_{j}=
\frac{\rho }{2}.
\end{align}
\end{lemma}
\begin{proof}
It is enough to prove the statement when $\rho =0$, since then the general
statement can then be deduced by considering $a'_{j}=a_{j}-\rho $. For
every~$n\geq 1$, defining~$S_{n}=\sum _{j=1}^{n} a_{j}$, we can write
\begin{equation*}
\sum _{j=1}^{n} j a_{j} = \sum _{1\leq i\leq j\leq n} a_{j} = \sum _{i=1}^{n}
(S_{n}-S_{i-1}) = nS_{n}-\sum _{i=1}^{n-1}S_{i} \,,
\end{equation*}
whence
\begin{equation*}
\frac {1}{n^{2}} \sum _{j=1}^{n} j a_{j} = \frac{S_{n}}{n} -
\frac {1}{n^{2}} \sum _{i=1}^{n-1}S_{i} \,.
\end{equation*}
The first term tends to~$0$ by our assumption~\eqref{eq:erg}. To bound
the second term, we write
\begin{align*}
\frac{1}{n^{2}}\sum _{i=0}^{n-1} \mathopen{\lvert }S_{i}\mathclose{\rvert } &\leq \frac{1}{n}\sum _{i=1}^{n-1}
\mathopen{\bigg \lvert }\frac{S_{i}}{i}\mathclose{\bigg \rvert },
\end{align*}
and then we observe that since $\mathopen{\lvert }S_{i}/i \mathclose{\rvert }$ converges to $0$, so do
the Ces\`aro means $\frac{1}{n}\sum _{i=1}^{n-1}\mathopen{\lvert }S_{i}/i \mathclose{\rvert }$, which concludes
the proof.
\end{proof}

\begin{proposition}%
\label{prop:det.explosive}
Let $\sigma \colon \mathbb{Z}\to \mathbb{N}\cup \{\mathfrak{s}\}$ be a
configuration of particles that contains at least one active particle and
satisfies
%
\begin{align}
\label{eq:det.explosive}
\lim _{n\to \infty}\frac {1} {n} \sum _{j=1}^{n}\mathopen{\lvert }\sigma (j)\mathclose{\rvert }=
\lim _{n\to \infty}\frac {1} {n} \sum _{j=-n}^{-1}\mathopen{\lvert }\sigma (j)\mathclose{\rvert }>
\rho _{c}.
\end{align}
Then it occurs with positive probability that all sites in
$\mathbb{Z}$ are visited when starting ARW from $\sigma $.
\end{proposition}
\begin{proof}
By shifting $\sigma $, we can take it to have an active particle at site~$0$
without loss of generality. Let $\rho $ be the value of the limits in
\eqref{eq:det.explosive}. By Lemma~\ref{lem:averages},
\begin{align*}
\lim _{n\to \infty}\frac{1}{n^{2}}\sum _{j=1}^{n} j\mathopen{\lvert }\sigma (j)\mathclose{\rvert }=
\frac{\rho }{2}.
\end{align*}
Thus, for $\epsilon =(\rho -\rho _{c})/2$ and $\beta = 2\rho $, there exists
some $N$ large enough that the conditions of Proposition~\ref{prop:right.avalanche}
hold, and we can conclude that
\begin{align*}
\mathbf{P}(X(n, \sigma )<\infty ) &\leq Ce^{-cn}
\end{align*}
for all $n\geq N$. By the same argument applied to $\sigma $ flipped across
the origin, there exists some $N'$ such that
\begin{align*}
\mathbf{P}(Y(n, \sigma )>-\infty ) &\leq Ce^{-cn}
\end{align*}
for all $n\geq N'$, with the same constants $c,C>0$.

Choose $m$ larger than $N$ and $N'$ and satisfying $Ce^{-cm}<1/2$. We stabilize
$\sigma $ on $\mathbb{Z}$ as follows. First, topple an active particle
at the origin until it falls asleep. Let $V_{0}\subseteq \mathbb{Z}$ be
the interval of sites it visits, and let $u_{0}$ be the system's odometer
after it sleeps. Conditional on $V_{0} \supseteq \mathopen{\llbracket }-m,m \mathclose{\rrbracket }$, the random
variables $X(m, \sigma , u_{0})$ and $Y(m, \sigma , u_{0})$ are distributed
as $X(m,\sigma )$ and $Y(m,\sigma )$, respectively, since starting the
instruction stacks at the indices given by $u_{0}$ does not alter their
distributions.

Now, the event $\{V_{0} \supseteq \mathopen{\llbracket }-m,m \mathclose{\rrbracket }\}$ holds with positive probability,
and
\begin{align*}
\mathbf{P}\bigl(X(m, \sigma , u_{0})<\infty \;\big \vert \;V_{0}
\supseteq \mathopen{\llbracket }-m,m \mathclose{\rrbracket }\bigr)&\leq Ce^{-cm}<1/2
\\
\intertext{and} \mathbf{P}\bigl(Y(m, \sigma , u_{0})>-\infty \;
\big \vert \;V_{0} \supseteq \mathopen{\llbracket }-m,m \mathclose{\rrbracket }\bigr)&\leq Ce^{-cm}<1/2.
\end{align*}
Thus by a union bound, it occurs with positive probability conditional
on the event $\{V_{0} \supseteq \mathopen{\llbracket }-m,m \mathclose{\rrbracket }\}$ that events
$\{X(m, \sigma , u_{0})=\infty \}$ and
$\{Y(m, \sigma , u_{0})=-\infty \}$ both occur. When all three events occur,
which occurs with positive probability, all sites in $\mathbb{Z}$ are visited
during the stabilization of $\sigma $, by definition of $X$ and $Y$.
\end{proof}

\begin{proof}[Proof of Theorem~\ref{thm:explosive}]
If $\rho <\rho _{c}$, then ARW fixates even starting from $\sigma $ with
all particles activated from the beginning by
\cite[Theorem~1]{RollaSidoraviciusZindy19}. Monotonicity of the odometer
when activating particles \cite[Lemma~2.5]{rolla2020activated} implies
that $\sigma $ is not explosive in this case.

If $\rho >\rho _{c}$, then by the ergodic theorem and Proposition~\ref{prop:det.explosive},
stabilizing the configuration $\sigma $ after activating the particles
at any nonvacant site visits all of $\mathbb{Z}$ with positive probability.
On this event, the odometer stabilizing $\sigma$ on $\mathbb{Z}$ is unchanged
by preemptively waking all particles in $\sigma$ by Lemma~\ref{lem:preemptive}.
But by \cite[Theorem~1]{RollaSidoraviciusZindy19}, that odometer is infinite
a.s., and so fixation does not occur.
\end{proof}




\newcommand{\arxivurl}[1]{\href{https://arXiv.org/abs/#1}{arXiv:#1}}


\begin{acks}
The authors thank the Erd\H{o}s
Center for its hospitality in hosting a workshop on Activated Random Walk
where this research was carried out.
\end{acks}

\end{document}